\RequirePackage{etoolbox}
\csdef{input@path}{{style/}{graphics/}}
\makeatletter
\input{arxiv-vmsta.cfg}
\makeatother
\documentclass[numbers,compress]{vmsta}

\volume{1}
\pubyear{2014}
\firstpage{129}
\lastpage{138}
\doi{10.15559/14-VMSTA14} 

\newtheorem{thm}{Theorem}
\newtheorem{lem}{Lemma}

\theoremstyle{definition}
\newtheorem{defin}{Definition}
\newtheorem{rem}{Remark}

\newcommand{\rrvert}{\vert}
\newcommand{\llvert}{\vert}
\newcommand{\lleft}{\left}
\newcommand{\rright}{\right}
\allowdisplaybreaks

\begin{document}

\begin{frontmatter}

\title{Heat equation with general stochastic measure colored in time}

\author{\inits{V.}\fnm{Vadym}\snm{Radchenko}\corref{cor1}}\email{vradchenko@univ.kiev.ua}
\cortext[cor1]{Corresponding author.}

\address{Department of Mathematical Analysis,\\
Taras Shevchenko National University of~Kyiv, Kyiv, Ukraine}

\markboth{V. Radchenko}{Heat equation with general stochastic measure
colored in time}

\begin{abstract}
A stochastic heat equation on $[0,T]\times{\mathbb R}$ driven by a
general stochastic measure $d\mu(t)$ is investigated in this paper. For
the integrator $\mu$, we assume the $\sigma$-additivity in probability
only. The existence, uniqueness, and H\"{o}lder regularity of the
solution are proved.
\end{abstract}


\received{2 October 2014}
\revised{10 November 2014}
\accepted{13 November 2014}
\publishedonline{5 December 2014}
\end{frontmatter}

\vspace*{-3pt}\section{Introduction}
\label{scintro}

In this paper, we consider a stochastic heat equation that can
formally be written as
\begin{equation}
\label{eqhshft} \lleft\{ %
\begin{array}{@{}l}
du(t, x)=a^2\displaystyle\frac{\partial^2 u(t, x)}{\partial x^2}\,dt +f\bigl(t, x, u(t,
x)\bigr)\,dt+ \sigma(t, x)\,d{\mu}(t) ,\\[6pt]
u(0, x)=u_0(x) ,
\end{array}\rright. %
\end{equation}
where $(t, x)\in[0, T]\times{\mathbb R}$, $a\in{\mathbb R}$, $a\ne0$,
and $\mu$ is a stochastic measure (SM) defined
on the Borel $\sigma$-algebra of $[0, T]$. We consider a solution to the
formal equation~(\ref{eqhshft}) in the mild
sense (see~Eq.~(\ref{eqhsift})). We prove the existence and
uniqueness of the solution and obtain H\"{o}lder
regularity of its paths under some general conditions for the
stochastic part of equation.\vadjust{\goodbreak}

A similar problem for $\mu$ dependent on the spatial variable $x$ was
considered in~\cite{rads09}. The stochastic heat
equation on fractals was studied in~\cite{rads12}, and a review of results on
equations driven by SMs is given
in~\cite{radspr14}.

For equations driven by white noise, the regularity of paths of
solutions was considered in~\cite[Chapter 3]{walsh}.
Equations driven by fractional noise were studied in~\cite[Chapter
2]{tudor13}. In many papers, the regularity of
solutions was considered in appropriate function spaces; see, for
example, \cite{krular} and references therein.

\section{Preliminaries}
\label{scprelimt}

Let ${\sf L}_0={\sf L}_0(\varOmega, {\mathcal F}, {\sf P} )$ be the set of
(equivalence classes of) all real-valued
random variables defined on a complete probability space $(\varOmega,
{\mathcal F}, {\sf P} )$. The convergence in ${\sf
L}_0$ is understood as the convergence in probability. Let ${\sf X}$ be
an arbitrary set, and ${\mathcal{B}}$ be a
$\sigma$-algebra of subsets of ${\sf X}$.

\begin{defin}
Any $\sigma$-additive mapping $\mu:\ {\mathcal{B}}\to{\sf L}_0$ is
called a {\em stochastic measure} (SM).
\end{defin}

In other words, $\mu$ is a vector measure with values in ${\sf L}_0$.
In~\cite{kwawoy}, such $\mu$ is called a general
SM.

Examples of SMs are the following. Let ${\sf X}=[0, T]\subset\mathbb
{R}_+$, ${\mathcal{B}}$ be the $\sigma$-algebra of
Borel subsets of $[0, T]$, and $N(t)$ be a square-integrable
martingale. Then $\mu({\sf A})=\int_0^T {\bf1}_{\sf
A}(t)\,dN(t)$ is an SM. If $W^H(t)$ is a fractional Brownian motion
with Hurst index~$H>1/2$ and $f: [0,
T]\to\mathbb{R}$ is a bounded measurable function, then $\mu({\sf
A})=\int_0^T f(t){\bf1}_{\sf A}(t)\,dW^H(t)$ is also
an SM, as follows from~\cite[Theorem~1.1]{memiva}. An $\alpha$-stable
random measure defined on a $\sigma$-algebra is
an SM \cite[Chapter 3]{samtaq}. Theorem~8.3.1 of~\cite{kwawoy} states
the conditions under which the increments of a
real-valued L\'{e}vy process generate an SM.

For a deterministic measurable function $g:{\sf X}\to{\mathbb R}$ and
SM~$\mu$, an integral of the form $\int_{\sf
X}g\,d\mu$ is defined and studied in~\cite[Chapter 7]{kwawoy}; see
also~\cite{curbera}. In particular, every bounded
measurable $g$ is integrable w.r.t. any~$\mu$. An analogue of the
Lebesgue dominated convergence theorem holds for
this integral \cite[Proposition~7.1.1]{kwawoy}.\looseness=1

We consider the {\em Besov spaces} $B^\alpha_{22}([c, d])$. Recall that
the norm in this classical space for $0<\alpha<
1$ may be introduced by
%
%
\begin{equation}
\label{eqbssn} \|g\|_{B^\alpha_{22}([c, d])}=\|g\|_{{\sf L}_{2}([c,
d])}+ \Biggl(\int
_0^{d-c}\ \bigl(w_2(g, r)
\bigr)^2 r^{-2\alpha-1} \,dr \Biggr)^{1/2},
\end{equation}
where
\[
w_2(g, r)=\sup_{0\le h\le r} \Biggl(\int
_{c}^{d-h} \bigl|g(s+h)-g(s)\bigr|^2\, ds
\Biggr)^{1/2}.
\]

For all $n\ge1$, $1\le k\le2^n$, put $\varDelta_{kn}^{(t)}= ((k-1)
2^{-n}t, k 2^{-n}t ]$.

The following estimate is a key tool for the proof of H\"{o}lder
regularity of the stochastic integral. In our estimates,
$C$ and $C(\omega)$ will denote a constant and a random constant,
respectively, which may be different from formula to
formula.

\begin{lem}[Lemma~3.2~\cite{rads09}]
Let SM~$\mu$ be defined on the Borel $\sigma$-algebra of $[0,t]$,\allowbreak  ${\sf
Z}$ be an arbitrary set, and $q(z,s): {\sf
Z}\times[0,t]\to{\mathbb R}$ be a function such that for some
$1/2<\alpha<1$ and for each $z\in{\sf Z}$,
$q(z,\cdot)\in B^\alpha_{22}([0,t])$. Then the random function
\[
\eta(z)=\int_{[0,t]}q(z,s)\,d\mu(s),\quad z\in{\sf Z},
\]
has a version $\tilde{\eta}(z)$ such that for some constant $C$
(independent of $z$, $\omega$) and each
$\omega\in\varOmega$,
%
%
\begin{align}
\bigl|\tilde{\eta}(z)\bigr|\le{}&\bigl|q(z,0)\mu \bigl([0,t] \bigr)\bigr|\nonumber\\
&{}+ C \bigl\|q(z, \cdot)\bigr\|_{B^\alpha_{22}([0,t])} \biggl\{\sum_{n\ge1}2^{n(1-2\alpha)}
\sum_{1\le k\le2^{n}} \bigl|\mu \bigl(\varDelta_{kn}^{(t)}
\bigr) \bigr|^2 \biggr\}^{1/2} .\label{eqvovt}
\end{align}
\end{lem}

From Lemma~3.1~\cite{rads09} it follows that, for $\varepsilon>0$,
%
%
\begin{equation}
\label{eqestm} \sum_{n\ge1}2^{-n\varepsilon}\sum
_{1\le k\le2^{n}} \bigl|\mu \bigl(\varDelta_{kn}^{(t)}
\bigr) \bigr|^2 < +\infty\quad {\rm a.\ s.}
\end{equation}

\section{The problem}
\label{scpropt}

Consider equation~(\ref{eqhshft}) in the following mild sense:
%
%
\begin{align}
u(t, x)={}&\int_{\mathbb R}{p}(t, x-y)u_0(y)\,dy
+\int_0^t\,ds \int_{\mathbb R}{p}(t-s, x-y)f \bigl(s, y, u(s, y) \bigr)\,dy\nonumber\\
&{}+\int_{(0,t]} \,d\mu(s) \int_{\mathbb R} {p}(t-s, x-y)\sigma(s, y)\,dy .\label{eqhsift}
\end{align}

Here
%
%
\begin{equation}
\label{eqfptx} {p}(t, x) = \frac{1}{2a\sqrt{\pi t}}\, e^{-\frac
{x^2}{4a^2t}}
\end{equation}
is the Gaussian heat kernel, $u(t, x)=u(t, x, \omega):[0, T]\times
{{\mathbb R}}\times\varOmega\to{\mathbb R}$ is an
unknown measurable random function, and $\mu$ is an SM defined on the Borel
$\sigma$-algebra of $[0,T]$. The integrals of
random functions w.r.t. $dy$ and $ds$ are taken for each fixed $\omega
\in\varOmega$.

Throughout this paper, we will use the following assumptions.

\renewcommand\theenumi{\textbf{A\arabic{enumi}.}}
\renewcommand\labelenumi{\theenumi}
\begin{enumerate}
\item $u_0(y)=u_0(y, \omega):{\mathbb R}\times\varOmega\to
{\mathbb
R}$ is measurable and $\omega$-wise bounded,\allowbreak
$\llvert u_0(y, \omega)\rrvert\le C(\omega)$.

\item $u_0(y)$ is H\"{o}lder continuous:
\[
\bigl\llvert u_0(y_1)-u_0(y_2)
\bigr\rrvert\le C(\omega)\llvert y_1-y_2\rrvert
^{\beta
(u_0)},\quad\beta(u_0)\ge1/2 .
\]

\item $f(s, y, v):[0, T]\times{{\mathbb R}}\times{\mathbb R}\to
{\mathbb R}$ is measurable and bounded: \mbox{$|f(s, y,
v)|\le C$}.

\item $f(s, y, v)$ is uniformly Lipschitz in $y, v \in{{\mathbb R}}$:
\[
\bigl\llvert f(s, y_1, v_1)-f(s, y_2,
v_2) \bigr\rrvert\le C \bigl(\llvert y_1-y_2
\rrvert+ \llvert v_1-v_2\rrvert \bigr) .
\]

\item $\sigma(s, y):[0, T]\times{{\mathbb R}}\to{\mathbb R}$ is
measurable and bounded: $\llvert\sigma(s,
y)\rrvert\le C$.

\item $\sigma(s, y)$ is H\"{o}lder continuous:
\[
\bigl\llvert\sigma(s_1, y_1)-\sigma(s_2,
y_2) \bigr\rrvert\le C \bigl(\llvert s_1-s_2
\rrvert^{\beta(\sigma)}+\llvert y_1-y_2\rrvert
^{\beta
(\sigma)} \bigr),\quad\beta(\sigma)>1/2.
\]

\item $\mu$ is H\"{o}lder continuous:
\[
\bigl\llvert\mu\bigl((s_1,s_2]\bigr) \bigr\rrvert\le C(\omega)
\llvert s_1-s_2\rrvert^{\beta(\mu)},\quad
s_1, s_2\in[0,T],\ \beta(\mu)>0.
\]
\end{enumerate}

Recall that $\int_{\mathbb R} {p}(t,x)\, dx =1$.

\section{H\"{o}lder continuity in $x$}
\label{scholdxt}

Consider the regularity of paths of the stochastic integral from~(\ref{eqhsift}).

\begin{lem}\label{lmhcvxt}
Let Assumptions A5 and A6 hold. Then, for any fixed $t\in[0, T]$ and
$\gamma_1<{\beta}(\sigma)-{1}/{2}$, the
stochastic function
\[
\vartheta(x)=\int_{(0,t]}\,d\mu(s) \int_{\mathbb R}
{p}(t-s, x-y)\sigma(s, y)\,dy ,\quad x\in{\mathbb R},
\]
has a H\"{o}lder continuous version with exponent $\gamma_1$.
\end{lem}

\begin{proof} Denote
\[
q(z,s)=\int_{\mathbb R} \bigl({p}(t-s, x_1-y)-{p}(t-s,
x_2-y) \bigr)\sigma(s, y)\, dy,\quad z=(x_1,x_2,t),
\]
and apply (\ref{eqvovt}) to $\eta(z)=\vartheta(x_1)-\vartheta(x_2)$. We
will estimate the Besov space norm
in~(\ref{eqvovt}). Consider the difference
%
%
\begin{align}
&q(z,s+h)-q(z,s)\nonumber\\
&\quad{}= \biggl(\int_{\mathbb R} {p}(t-s-h, x_1-y)
\sigma(s+h,y)\,dy-\int_{\mathbb R} {p}(t-s, x_1-y)
\sigma(s,y)\,dy \biggr)\nonumber\\
&\qquad{}- \biggl(\int_{\mathbb R} {p}(t-s-h, x_2-y)
\sigma(s+h,y)\,dy-\int_{\mathbb R} {p}(t-s, x_2-y)
\sigma(s,y)\,dy \biggr)\nonumber\\
&\quad{}:=D_1-D_2.\label{eqdone}
\end{align}
Using~(\ref{eqfptx}) and the change of variables
\[
v=\frac{x_1-y}{2a\sqrt{t-s-h}},\qquad v=\frac{x_1-y}{2a\sqrt{t-s}},
\]
we get
%
%
\begin{align}
|D_1|&= C \Biggl|\int_{\mathbb R} e^{-v^2}
\sigma(s+h,x_1-2av\sqrt{t-s-h})\, dv\nonumber\\
&\quad{}-\int_{\mathbb R}e^{-v^2}\sigma(s,x_1-2av
\sqrt{t-s})\,dv \Biggr|\nonumber\\
&\stackrel{\rm A6} {\le} C \int_{\mathbb R} e^{-v^2}
\bigl(|h|^{\beta (\sigma)}
+ \bigl|v (\sqrt{t-s-h}-\sqrt{t-s} )\bigr|^{\beta(\sigma)} \bigr)\,dv \nonumber\\
&=C\int_{\mathbb R}e^{-v^2} \biggl(|h|^{\beta(\sigma)}+
\frac{|v|^{\beta
(\sigma)}|h|^{\beta(\sigma)}}{
|\sqrt{t-s-h}+\sqrt{t-s} |^{\beta(\sigma)}} \biggr)\,dv\nonumber\\
&\le C|h|^{\beta(\sigma)}\int_{\mathbb R}e^{-v^2} \biggl(1+
\frac{|v|^{\beta(\sigma)}}{\sqrt{t-s}^{\beta(\sigma)}} \biggr)\,dv
=C|h|^{\beta(\sigma)}(t-s)^{-\beta(\sigma)/2}.\label{eqdonest}
\end{align}
By a similar way, we can estimate $|D_2|$ and obtain
%
%
\begin{equation}
\label{eqquhh} \bigl|q(z,s+h)-q(z,s) \bigr|\le C |h|^{\beta(\sigma)}(t-s)^{-\beta
(\sigma)/2}.
\end{equation}

Further, consider
\begin{align*}
q(z,s+h)-q(z,s)={}& \biggl(\int_{\mathbb R} {p}(t-s-h,x_1-y)\sigma(s+h,y)\, dy\\
&{}-\int_{\mathbb R} {p}(t-s-h, x_2-y)\sigma(s+h,y)\,dy\biggr)\\
&{}- \biggl(\int_{\mathbb R}{p}(t-s, x_1-y)\sigma(s,y)\,dy\\
&{}-\int_{\mathbb R} {p}(t-s, x_2-y)\sigma(s,y)\,dy\biggr) :=E_1-E_2.
\end{align*}
Using~(\ref{eqfptx}) and the substitutions
\[
v=\frac{x_1-y}{2a\sqrt{t-s-h}},\quad v=\frac{x_2-y}{2a\sqrt{t-s-h}},
\]
we get
\begin{align*}
|E_1|&= C \Biggl|\int_{\mathbb R} e^{-v^2}
\sigma(s+h,x_1-2av\sqrt{t-s-h})\, dv\\
&\quad{}-\int_{\mathbb R} e^{-v^2}\sigma(s+h,x_2-2av
\sqrt{t-s-h})\,dv \Biggr|\\
&\stackrel{\rm A6} {\le} C \int_{\mathbb R} e^{-v^2}|x_1-x_2|^{\beta(\sigma)}
\,dv =C|x_1-x_2|^{\beta(\sigma)}.
\end{align*}
Similarly, we can estimate $|E_2|$ (we consider $|E_1|$ for $h=0$)
and obtain
%
%
\begin{equation}
\label{eqquht} \bigl|q(z,s+h)-q(z,s) \bigr|\le C |x_1-x_2|^{\beta(\sigma)}.
\end{equation}

The product of~(\ref{eqquhh}) raised to the power~$\lambda$ and~(\ref
{eqquht}) raised to the power~$1-\lambda$,
$0<\lambda<1$, now satisfies
\begin{align*}
\bigl|q(z,s+h)-q(z,s)\bigr|&\le C|h|^{\lambda\beta(\sigma)}{(t-s)}^{-\beta(\sigma
)\lambda/2}
|x_1-x_2|^{(1-\lambda)\beta(\sigma)},\\
w_2 \bigl(q(z,\cdot),r \bigr)&\le Cr^{\lambda\beta(\sigma)}
|x_1-x_2|^{(1-\lambda
)\beta(\sigma)}.
\end{align*}
If $\lambda\beta(\sigma)>1/2$, then the integral from~(\ref{eqbssn}) is
finite for some $\alpha>1/2$. In this case, the
integral does not exceed $C|x_1-x_2|^{(1-\lambda)\beta(\sigma)}$.

From the estimate of $E_1$ for $h=0$ we obtain
\[
\bigl|q(z,0)\bigr| \le C |x_1-x_2|^{\beta(\sigma)},\qquad
\bigl\|q(z,\cdot)\bigr\|_{{\sf L}_{2}([0, t])} \le C |x_1-x_2|^{\beta(\sigma)}.
\]
Therefore, we have
\[
\bigl|\vartheta(x_1)-\vartheta(x_2)\bigr|\le C(
\omega)|x_1-x_2|^{\gamma_1},\quad
\gamma_1={(1-\lambda)\beta(\sigma)}.
\]
Under the restriction $\lambda\beta(\sigma)>1/2$, we can get any $\gamma
_1<{\beta}(\sigma)-{1}/{2}$.
\end{proof}

\section{H\"{o}lder continuity in $t$}
\label{scholdtt}

\begin{lem}\label{lmhcvtt}
Assume that Assumptions A5, A6, and A7 hold. Then, if  $\gamma_2\le\beta(\mu)$ and $\gamma_2<\beta(\sigma)-1/2$, then for any fixed $x\in
{\mathbb R}$,  the stochastic process
\[
\bar{\vartheta}(t)=\int_{(0,t]} \,d\mu(s) \int
_{\mathbb R} {p}(t-s, x-y)\sigma(s, y)\,dy,\quad t\in[0,T],
\]
has a H\"{o}lder continuous version with exponent $\gamma_2$.
\end{lem}

\begin{proof}
For $t_1<t_2$, we have
\begin{align*}
\bar{\vartheta}(t_2)-\bar{\vartheta}(t_1)&=\int
_{(t_1,t_2]}\, d\mu(s) \int_{\mathbb R}
{p}(t_2-s, x-y)\sigma(s, y)\,dy\\
&\quad{}+\int_{(0,t_1]}\, d\mu(s) \int_{\mathbb R}
\bigl({p}(t_2-s, x-y)-{p}(t_1-s, x-y) \bigr)\sigma(s, y)\,dy\\
&:=F_1+F_2.
\end{align*}
\textbf{Step 1. Estimation of $\boldsymbol{F_1}$.} Consider segments $[0,T]$, $\varDelta
_{kn}^{(T)}= ((k-1) 2^{-n}T,  k
2^{-n}T ]$, and the function
\[
\bar{q}(z,s)=\int_{\mathbb R} {p}(t_2-s, x-y)\sigma(s,
y)\,dy,\quad s\in[t_1,t_2],\ z=(x,t_2).
\]
From the estimates of $D_1$ in~(\ref{eqdone}) and~(\ref{eqdonest}) it
follows that
%
%
\begin{equation}
\label{eqeghx} \bigl|\bar{q}(z,s+h)-\bar{q}(z,s)\bigr|
\le C|h|^{\beta(\sigma)}(t_2-s)^{-\beta(\sigma)/2},
\quad s\in[t_1-h,t_2-h).
\end{equation}

Let $k_{n1}$ and $k_{n2}$ be such that $t_1\in\varDelta_{k_{n1}n}^{(T)}$,
$t_2\in\varDelta_{k_{n2}n}^{(T)}$. For the
functions
\[
\bar{q}_n(z,s)=\sum_{k=1}^{2^n}
\bar{q} \bigl(z,(k-1) 2^{-n}T\vee t_1\wedge t_2
\bigr){\bf1}_{\varDelta^{(T)}_{kn}}(s),
\]
the analogue of the Lebesgue theorem~\cite[Proposition 7.1.1]{kwawoy}
implies that
%
%
\begin{align}
&\Biggl|\int_{(t_1,t_2]}\bar{q}(z,s)\, d\mu(s) \Biggr|
= \Biggl|{\rm p}\lim_{n\to\infty}
\int_{(t_1,t_2]}\bar{q}_n(z,s) \,d\mu(s)\Biggr |\nonumber\\
&\!\!\!\quad{}= \Biggl|\int_{(t_1,t_2]}\bar{q}_0(z,s)\, d\mu(s)
+\sum_{n=1}^{\infty} \biggl(\int_{(t_1,t_2]}
\bar{q}_{n}(z,s)\, d\mu(s)- \int_{(t_1,t_2]}
\bar{q}_{n-1}(z,s)\, d\mu(s) \biggr) \Biggr|\nonumber\\
&\!\!\!\quad{}\le\bigl|\bar{q}(z,t_1)\mu \bigl((t_1,t_2]\bigr) \bigr|+
\sum_{n=n_0}^{\infty} \bigl|\bigl(\bar{q}
\bigl(z,k_{n1}2^{-n}T \bigr)-\bar{q}(z,t_1)\bigr)
\mu \bigl(\varDelta^{(T)}_{(k_{n1}+1)n} \bigr) \bigr|\nonumber\\
&\!\!\!\qquad{}+\sum_{n=n_0}^{\infty}\sum
_{k: k_{n1}\le2k-2< k_{n2}-2} \bigl| \bigl(\bar{q} \bigl(z,(2k-1)2^{-n}T \bigr)-
\bar{q} \bigl(z,(2k-2)2^{-n}T \bigr) \bigr) \mu \bigl(
\varDelta^{(T)}_{(2k)n} \bigr) \bigr|\nonumber\\
&\!\!\!\qquad{}+\sum_{n=n_0}^{\infty} \bigl| \bigl(\bar{q}
\bigl(z,(k_{n2}-1)2^{-n}T \bigr)-\bar{q} \bigl(z,(k_{n2}-2)2^{-n}T
\bigr) \bigr) \mu\bigl( \bigl((k_{n2}-1)2^{-n}T,t_2\bigr]
\bigr) \bigr|\label{eqtott}.
\end{align}
Here $n_0$ is such that
%
%
\begin{equation}
\label{eqtttn} 2^{-n_0}T<t_2-t_1
\le2^{-n_0+1}T.
\end{equation}
We have
%
%
\begin{equation}
\label{eqknkn} k_{n2}-k_{n1}\le(t_2-t_1)2^{n+2}/T,
\quad n\ge n_0.
\end{equation}

Applying Assumptions A5, A7, (\ref{eqeghx}), and the Cauchy inequality from~(\ref
{eqtott}) for $0<\varepsilon<2\beta(\sigma)-1$, we
obtain
\begin{align*}
&\Biggl|\int_{(t_1,t_2]}\bar{q}(z,s) \,d\mu(s)\Biggr |\\
&\quad{}\le C(\omega)
(t_2-t_1)^{\beta(\mu)} +C(\omega)\sum
_{n=n_0}^{\infty}2^{-n\beta(\mu)}\\
&\qquad{}+C\sum_{n=n_0}^{\infty}\sum
_{k: k_{n1}\le2k-2< k_{n2}-2} 2^{-n\beta(\sigma)} \bigl(t_2-(2k-2)2^{-n}T
\bigr)^{-\beta(\sigma)/2} \bigl|\mu \bigl(\varDelta^{(T)}_{(2k)n} \bigr) \bigr|\\
&\qquad{}+C(\omega)\sum_{n=n_0}^{\infty} 2^{-n\beta(\mu)}\\
&\quad{}\le C(\omega) (t_2-t_1)^{\beta(\mu)}+C(
\omega)2^{-n_0\beta(\mu)}\\
&\qquad{}+C \Biggl(\sum_{n=n_0}^{\infty}
2^{\varepsilon n}2^{-2n\beta(\sigma)}\sum_{k: k_{n1}\le2k-2< k_{n2}-2}
\bigl(t_2-(2k-2)2^{-n}T \bigr)^{-\beta(\sigma)}
\Biggr)^{1/2}\\
&\qquad {}\times \Biggl(\sum_{n=0}^{\infty}2^{-\varepsilon n}
\sum_{k=1}^{2^n} \mu^2 \bigl(
\varDelta^{(T)}_{kn} \bigr) \Biggr)^{1/2}\\
&\stackrel{\eqref{eqestm},\eqref{eqtttn}} {\le} C(\omega)
(t_2-t_1)^{\beta(\mu)}\\
&\qquad{}+C(\omega) \Biggl(\sum_{n=n_0}^{\infty}
2^{\varepsilon n}2^{-2n\beta
(\sigma)}\sum_{1\le i< (k_{n2}-k_{n1})/2}
\bigl(i2^{-n}T \bigr)^{-\beta(\sigma)} \Biggr)^{1/2}\\
&\quad{}\le C(\omega) (t_2-t_1)^{\beta(\mu)} +C(\omega)
\Biggl(\sum_{n=n_0}^{\infty
}2^{n(\varepsilon-\beta(\sigma))} (
k_{n2}-k_{n1} )^{1-\beta(\sigma)} \Biggr)^{1/2}\\
&\quad{}\stackrel{\eqref{eqknkn}} {\le} C(\omega) (t_2-t_1)^{\beta(\mu)}+C(
\omega) (t_2-t_1)^{(1-\beta(\sigma))/2}2^{-n_0(2\beta(\sigma
)-\varepsilon
-1)/2}
\\
&\quad{}\stackrel{\eqref{eqtttn}} {\le} C(\omega) (t_2-t_1)^{\beta(\mu)}
+C(\omega) (t_2-t_1)^{(\beta(\sigma)-\varepsilon)/2} \le C(\omega)
(t_2-t_1)^{\gamma_2}.
\end{align*}
\textbf{Step 2. Estimation of $\boldsymbol{F_2}$.} Now denote
\begin{align*}
&\tilde{q}(z,s)=\int_{\mathbb R} \bigl({p}(t_2-s,
x-y)-{p}(t_1-s, x-y) \bigr)\sigma(s, y)\,dy,\\
&\quad s\in[0,t_1],\ z=(x,t_1,t_2).
\end{align*}
Using the change of variables
\[
v=\frac{x-y}{2a\sqrt{t_2-s}},\qquad v=\frac{x-y}{2a\sqrt{t_1-s}},
\]
we get
%
%
\begin{align}
\bigl|\tilde{q}(z,s)\bigr|={}&C \Biggl|\int_{\mathbb R} e^{-v^2}
\sigma(s,x-2av\sqrt{t_2-s})\,dv
\nonumber\\
&{}-\int_{\mathbb R} e^{-v^2}\sigma(s,x-2av
\sqrt{t_1-s})\,dv\Biggr |
\nonumber\\
\stackrel{\rm A6} {\le}{}& C \int_{\mathbb R} e^{-v^2} \bigl|v (
\sqrt{t_2-s}-\sqrt{t_1-s} )\bigr|^{\beta(\sigma)}\,dv
\nonumber\\
\stackrel{\eqref{eqdonest}} {\le}{}&C(t_2-t_1)^{\beta(\sigma
)}(t_2-s)^{-\beta(\sigma)/2}.\label{eqdonestt}
\end{align}
Also, analogously to (\ref{eqdone}) and (\ref{eqdonest}), we have
%
%
\begin{equation}
\label{eqtilq} \bigl|\tilde{q}(z,s+h)-\tilde{q}(z,s) \bigr|\le C|h|^{\beta(\sigma
)}(t_1-s)^{-\beta(\sigma)/2}.
\end{equation}
From~(\ref{eqdonestt}) and (\ref{eqtilq}) for $0<\lambda<1$ and $0\le s\le
t_1-h$, we obtain
\begin{align*}
&\bigl|\tilde{q}(z,s+h)-\tilde{q}(z,s)\bigr|\\
&\quad{}\le C|h|^{\lambda\beta(\sigma)}{(t_1-s)}^{-\beta(\sigma)\lambda/2}
(t_2-t_1)^{(1-\lambda)\beta(\sigma)}(t_2-s-h)^{-(1-\lambda)\beta(\sigma)/2},\\
&w_2 \bigl(\tilde{q}(z,\cdot),r \bigr)\le Cr^{\lambda\beta(\sigma)}
(t_2-t_1)^{(1-\lambda)\beta(\sigma)}.
\end{align*}
If $\lambda\beta(\sigma)>1/2 \Leftrightarrow(1-\lambda)\beta(\sigma
)<{\beta}(\sigma)-{1}/{2}$, then the integral
from~(\ref{eqbssn}) is finite for some $\alpha>1/2$. In this case, the
integral does not exceed
$C(t_2-t_1)^{(1-\lambda)\beta(\sigma)}$.

From~(\ref{eqdonestt}) we get
\[
\bigl|\tilde{q}(z,0)\bigr| \le C (t_2-t_1)^{\beta(\sigma)/2},\qquad
\bigl\|\tilde{q}(z,\cdot)\bigr\|_{{\sf L}_{2}([0, t_1])} \le C (t_2-t_1)^{\beta
(\sigma)/2}.
\]
Therefore, from~(\ref{eqvovt}) we have $ |F_2|\le C(\omega
)(t_2-t_1)^{\gamma_2}$, which finishes the proof.
\end{proof}

\section{Solution to the equation}
\label{scsolutiont}

\begin{thm} Suppose that Assumptions A1--A6 hold.
\renewcommand\theenumi{\arabic{enumi}.}
\renewcommand\labelenumi{\theenumi}
\begin{enumerate}
\item Equation~(\ref{eqhsift}) has a solution $u(t, x)$. If $v(t, x)$ is
another solution to~(\ref{eqhsift}), then for
all $t$ and $x$, $u(t, x)=v(t, x)$~a.s.\vadjust{\goodbreak}
\item For any fixed $t\in[0, T]$ and $\gamma_1<{\beta}(\sigma)-1/2$, the
stochastic function $u(t, x)$, $x\in\mathbb{R}$,
has a H\"{o}lder continuous version with exponent $\gamma_1$.
\item In addition, let Assumption A7 hold. Then for any fixed $\delta
>0$ and \(\gamma_1\),~\(\gamma_2\) such that
$\gamma_1<{\beta}(\sigma)-1/2$,
 $\gamma_2\le\beta(\mu)$, and
 $\gamma_2<\beta(\sigma)-1/2$, the
stochastic function $u(t, x)$ has a version
$\tilde{u}(t, x)$ such that
%
%
\begin{align}
\label{equtxh}
&\bigl\llvert\tilde{u}(t_1, x_1)-
\tilde{u}(t_2, x_2) \bigr\rrvert\le C(\omega) \bigl(
\llvert t_1-t_2\rrvert^{\gamma_2}+ \llvert
x_1-x_2\rrvert^{\gamma_1} \bigr),
\nonumber
\\
&\quad{}t_i\in[\delta, T],\ x_i\in\mathbb{R}.
\end{align}
\end{enumerate}
\end{thm}

\begin{proof}
Consider the standard iteration process. Take $u^{(0)}(t, x)=0$ and set
\begin{align*}
u^{(n+1)}(t, x)&=\int_{\mathbb R}{p}(t, x-y)u_0(y)
\,dy\\
&\quad{} +\int_0^t \,ds \int_{\mathbb R}{p}(t-s,x-y)f
\bigl(s, y, u^{(n)}(s, y) \bigr)\,dy\\
&\quad{}+\int_{(0,t]} \,d\mu(s) \int_{\mathbb R} {p}(t-s,x-y)\sigma(s, y)\,dy .
\end{align*}

Further, we can repeat the proof of Theorem~\cite{rads09}. Instead of
reference to Lemmas~5.1 and 6.1 of~\cite{rads09}, we
can refer to Lemmas \ref{lmhcvxt} and \ref{lmhcvtt} of this paper.
\end{proof}

\begin{rem}
For $u$, we obtained less regularity than  for elements of
equation~(\ref{eqhsift}). However, a solution to a heat equation usually
has the same regularity or even more regular than the coefficients.
One may expect that using other
methods gives~(\ref{equtxh}) with exponents $\gamma_2\le\beta(\mu
)\wedge\gamma_2<\beta(\sigma)$ and
$\gamma_1<{\beta}(\sigma)$.
\end{rem}

\end{document}